\documentclass{amsart} 
\usepackage{amsmath,amssymb,amsthm,amsfonts,amsxtra,mathtools}
\usepackage{enumerate,verbatim,mathrsfs,comment,color,url}
\usepackage[all,2cell,ps]{xy}
\usepackage[pagebackref]{hyperref}
\usepackage{graphicx}
\usepackage{verbatim}

\DeclareMathOperator{\ann}{ann}

\DeclareMathOperator{\depth}{depth}

\DeclareMathOperator{\Ext}{Ext}

\DeclareMathOperator{\gdim}{G-dim}

\DeclareMathOperator{\grade}{grade}
\DeclareMathOperator{\Hom}{Hom}

\DeclareMathOperator{\id}{id}

\DeclareMathOperator{\pd}{pd}

\DeclareMathOperator{\Soc}{Soc}
\DeclareMathOperator{\Supp}{Supp}
\DeclareMathOperator{\Tor}{Tor}
\DeclareMathOperator{\Tot}{Tot}

\DeclareMathOperator{\type}{type}

\renewcommand{\ge}{\geqslant}
\renewcommand{\le}{\leqslant}

\newcommand{\bz}{\mathbb{Z}}

\newcommand{\fm}{\mathfrak{m}}

\renewcommand{\iff}{if and only if }

\theoremstyle{plain}
\newtheorem{theorem}{Theorem}[section]
\newtheorem{lemma}[theorem]{Lemma}
\newtheorem{proposition}[theorem]{Proposition}
\newtheorem{corollary}[theorem]{Corollary}

\newtheorem{innercustomtheorem}{Theorem}
\newenvironment{customtheorem}[1]
  {\renewcommand\theinnercustomtheorem{#1}\innercustomtheorem}
  {\endinnercustomtheorem}

 \newtheorem{innercustomcorollary}{Corollary}
 \newenvironment{customcorollary}[1]
 {\renewcommand\theinnercustomcorollary{#1}\innercustomcorollary}
 {\endinnercustomcorollary}
  
\theoremstyle{definition}
\newtheorem{definition}[theorem]{Definition}

\newtheorem{example}[theorem]{Example}

\newtheorem{para}[theorem]{}
\newtheorem{question}[theorem]{Question}
\newtheorem{setup}[theorem]{Setup}

\theoremstyle{remark}
\newtheorem{remark}[theorem]{Remark}

\numberwithin{equation}{section}

\title[Ext, Tate cohomology and Gorenstein rings]{Gorenstein rings via homological dimensions, and symmetry in vanishing of Ext and Tate cohomology}
\author[D.~Ghosh]{Dipankar Ghosh}
\address{Department of Mathematics, Indian Institute of Technology Kharagpur, West Bengal - 721302, India}
\email{dipankar@maths.iitkgp.ac.in, dipug23@gmail.com}

\author[T.J.~Puthenpurakal]{Tony J. Puthenpurakal}
\address{Department of Mathematics, Indian Institute of Technology Bombay, Powai, Mumbai 400076, India}
\email{tputhen@math.iitb.ac.in}

\date{February 17, 2023\\
Corresponding Author: Dipankar Ghosh}

\begin{document}

\pagenumbering{arabic}
\thispagestyle{empty}
 
 \begin{abstract}
 	The aim of this article is to consider the spectral sequences induced by tensor-hom adjunction, and provide a number of new results. Let $R$ be a commutative Noetherian local ring of dimension $d$. In the 1st part, it is proved that $R$ is Gorenstein if and only if it admits a nonzero CM (Cohen-Macaulay) module $M$ of finite Gorenstein dimension $g$ such that $\type(M) \le \mu( \Ext_R^g(M,R) )$ (e.g., $\type(M)=1$). This considerably strengthens a result of Takahashi. Moreover, we show that if there is a nonzero $R$-module $M$ of depth $\ge d - 1$ such that the injective dimensions of $M$, $\Hom_R(M,M)$ and $\Ext_R^1(M,M)$ are finite, then $M$ has finite projective dimension and $R$ is Gorenstein. In the 2nd part, we assume that $R$ is CM with a canonical module $\omega$. For CM $R$-modules $M$ and $N$,
 	we show that the vanishing of one of the following implies the same for others: $\Ext_R^{\gg 0}(M,N^{+})$, $\Ext_R^{\gg 0}(N,M^{+})$ and $\Tor_{\gg 0}^R(M,N)$, where $M^{+}$ denotes $\Ext_R^{d-\dim(M)}(M,\omega)$. This strengthens a result of Huneke and Jorgensen. Furthermore, we prove a similar result for Tate cohomologies under the additional condition that $R$ is Gorenstein.
 \end{abstract}
\maketitle

\section{Introduction}
	
	Unless otherwise specified, throughout $(R,\fm,k)$ is a commutative Noetherian local ring.
	All $R$-modules are assumed to be finitely generated.
	
	The aim of this article is to consider the spectral sequences induced by tensor-hom adjunction, and prove a number of new results on the following topics: (I) Criteria for a local ring to be Gorenstein in terms of Gorenstein and injective dimensions of certain modules. (II) Symmetry in vanishing of Ext and Tate cohomology over CM (Cohen-Macaulay) and Gorenstein local rings respectively.
	
	(I) Our results on the 1st topic are motivated by the following characterizations.
	
	\begin{theorem}[Auslander-Bridger, Takahashi, Celikbas--Sather-Wagstaff] \label{thm:known-charac-via-G-dim}
		The ring $R$ is Gorenstein if and only if at least one of the following holds true.
		\begin{enumerate}[\rm (1)]
			\item $\gdim_R(k) < \infty$, see, e.g., \cite[(1.4.9)]{Cr}.
			\item \cite[(4.20)]{AB69} $\gdim_R(M) < \infty$ for every $R$-module $M$.
			\item \cite[Thm.~6.5]{Tak06} For some $0 \le n \le \depth(R)+2$, the $n$th syzygy module $\Omega_R^n(k)$ has a nonzero direct summand of G-dimension $0$.
			\item \cite[Thm.~4.8]{CS16} $R$ has an integrally closed ideal $I$ such that $\depth(R/I) = 0$ and $\gdim_R(I) < \infty$. See \cite[Cor.~5.5.(4)]{GS22} for a stronger result.
		\end{enumerate}
	\end{theorem}
	
	Various local rings can be characterized in terms of injective dimension of certain modules. For example, $R$ is regular \iff the residue field $k$ has finite injective dimension, cf.~\cite[3.1.26]{BH93}. It is generalized further in \cite[Thm.~3.7]{GGP} that $R$ is regular if and only if some syzygy $\Omega_R^n(k)$ $(n \ge 0)$	has a nonzero direct summand of finite injective dimension. The Bass'	conjecture (known by \cite{PS73, Rob87}) says that $R$ is CM if it admits a nonzero $($finitely generated$)$ module of finite injective dimension. For Gorenstein local rings, the following are known.
	
	\begin{theorem}[Peskine-Szpiro and Foxby]{~}\\ \label{thm:known-charac-via-inj-dim}
		The ring $R$ is Gorenstein \iff at least one of the following holds true.
		\begin{enumerate}[\rm (1)]
			\item \cite[Ch.II, (5.5)]{PS73}
			$R$ admits a nonzero cyclic module of finite injective dimension.
			\item \cite{Fo77} $R$ admits a nonzero module $M$ such that both injective dimension $\id_R(M)$ and projective dimension $\pd_R(M)$ are finite.
			
		\end{enumerate}
	\end{theorem}

	Celikbas and Sather-Wagstaff deduced Theorem~\ref{thm:known-charac-via-G-dim}.(4) by proving that if $M$ is a $\gdim$-test $R$-module (Definition~\ref{defn:G-dim-test}) such that $\Ext^i_R(M, R) = 0$ for all $ i \gg 0 $, then $R$ is Gorenstein; see \cite[Thm.~4.4]{CS16}. Their proof relies on the use of derived categories. In the present study, we give a very short and elementary proof of this result under the assumption that $R$ is CM, see Theorem~\ref{thm:G-dim-test-Gor-charac}.
%
%
%
	In Section~\ref{sec:charac-Gor-via-G-dim}, we study G-perfect modules (Definition~\ref{defn:grade-and-G-perfect}). In Theorem~\ref{thm:CM-iff-G-perfect-for-finite-G-dim}, we show that a CM $R$-module of finite G-dimension is always G-perfect, and the converse is also true when $R$ is CM. Moreover, we prove the following:
	
	\begin{customtheorem}{\ref{thm:G-perfect-and-type-formula}}
		If $M$ is a G-perfect $R$-module of G-dimension $g$, then
		\[
			\type(M) = \mu(\Ext_R^g(M,R)) \type(R),
		\]
		where $\mu(L)$ denotes the minimal number of generators of an $R$-module $L$.
	\end{customtheorem}

	As a consequence of Theorems~\ref{thm:CM-iff-G-perfect-for-finite-G-dim} and \ref{thm:G-perfect-and-type-formula}, we obtain the following characterization of Gorenstein local rings in terms of G-dimension of certain modules.
	
\begin{customcorollary}{\ref{cor:Gor-iff-there-is-CM-G-dim-finite-mod}}
	The  following statements are equivalent:
	\begin{enumerate}[\rm (1)]
		\item $R$ is Gorenstein.
		\item $R$ admits a nonzero CM module $M$ of finite G-dimension $g$ such that
		\begin{center}
			$\type(M) \le \mu( \Ext_R^g(M,R) )$ \quad {\rm (}e.g., $\type(M)=1${\rm )}.
		\end{center}
	\end{enumerate}
\end{customcorollary}
	
	In \cite[Thm.~2.3]{Ta04}, Takahashi showed that $R$ is Gorenstein if and only if there is a CM $R$-module of finite G-dimension and of type $1$. Corollary~\ref{cor:Gor-iff-there-is-CM-G-dim-finite-mod} considerably strengthens this result. It also provides a positive answer to the following question under the additional assumption that $\type(M) \le \mu( \Ext_R^g(M,R) )$, where $g = \gdim_R(M)$.

\begin{question}[Christensen]\cite[p.~40]{Cr}\label{ques:CM-G-dim-is-R-CM}
	If there exists a nonzero CM $R$-module $M$ of finite G-dimension, then is $R$ CM?
\end{question}

Question~\ref{ques:CM-G-dim-is-R-CM} has an affirmative answer when $\pd_R(M)$ is finite, which is a consequence of the intersection theorem, see, e.g., \cite[p.~40]{Cr}. The intersection theorem does not extend to G-dimension, as noted in \cite[(1.5.7)]{Cr}. However, there are some partial positive answers to Question~\ref{ques:CM-G-dim-is-R-CM}, see \cite[Thm.~2.3]{Ta04}, \cite[Cor.~4.4 and 4.5]{DMT14} and \cite[Thm.~3.6]{DMTT20}.

In view of the results in Theorem~\ref{thm:known-charac-via-inj-dim}, we are interested to find some classes of $R$-modules such that having finite injective dimensions of these modules ensure that $R$ is Gorenstein. In this theme, we prove Theorem~\ref{thm:injdim-Ext-finite-N-d-1}, which also says that a nonzero module of finite injective dimension can be used as a test module which detects the finiteness of projective dimension of a given module.

\begin{customtheorem}{\ref{thm:injdim-Ext-finite-N-d-1}}
	Let $N$ be a nonzero $R$-module such that $\depth(N) \ge \dim(R) - 1$ and $\id_R(N) < \infty$. For an $R$-module $M$, if $\id_R( \Ext_R^i(M,N) )$ is finite for all $i$, then $\pd_R(M)$ is finite.
\end{customtheorem}

As a consequence, we obtain the following characterization of Gorenstein local rings in terms of finite injective dimensions of certain modules.

\begin{customcorollary}{\ref{cor:injdim-Ext-finite-N-d-1-charac-Gor}}
	If there is a nonzero $R$-module $M$ of depth $\ge \dim(R) - 1$ such that the injective dimensions of $M$, $\Hom_R(M,M)$ and $\Ext_R^1(M,M)$ are finite, then $M$ has finite projective dimension and $R$ is Gorenstein.
\end{customcorollary}

(II)
In the 2nd part of this article, we investigate symmetry in vanishing of Ext and Tate cohomology, and their relations with the vanishing of Tor and Tate homology respectively. These were first studied over complete intersection rings by Avramov-Buchweitz \cite{AB00}. See \ref{Tate-co-homology} for the definition of Tate (co)homology. For $R$-modules $M$ and $N$, the vanishing of
\begin{center}
	$\Ext^{\gg 0}_R(M,N)$ (resp., $\Ext^{\ll 0}_R(M,N)$, $\Ext^{\star}_R(M,N)$)
\end{center}
means that $\Ext^i_R(M,N) = 0$ for all $i\gg 0$ (resp., $i\ll 0$, $i\in \bz$). When $R$ is a local complete intersection ring, Avramov-Buchweitz proved that the following are equivalent:\\
(1) $\Ext^{\gg 0}_R(M,N) = 0$, (2) $\Ext^{\gg 0}_R(N,M) = 0$, (3) $\Tor_{\gg 0}^R(M,N) = 0$,\\
(4) $\widehat{\Ext}_R^{\ll 0}(M,N) = 0$, (5) $\widehat{\Ext}_R^{\ll 0}(N,M) = 0$, (6) $\widehat{\Tor}_{\ll 0}^R(M, N) = 0$,\\
(7) $\widehat{\Ext}_R^{\star}(M,N) = 0$, (8) $\widehat{\Ext}_R^{\star}(N,M) = 0$, and (9) $\widehat{\Tor}_{\star}^R(M, N) = 0$,\\
see \cite[Thms.~4.7, 4.9 and 6.1]{AB00}. In \cite[Thm.~4.1]{HJ03}, Huneke-Jorgensen showed that the symmetry in vanishing of Ext (i.e., the equivalence of (1) and (2)) holds for a new class of rings, namely AB rings (\cite[Defn.~3.1]{HJ03}), which lie (strictly) between complete intersection and Gorenstein rings (cf.~\cite[Thm.~3.6]{HJ03} and \cite[pp.~471, Thm]{JS04}). However, the symmetry in vanishing of Ext does not hold for all Gorenstein rings, due to \cite[Thm.~1.2]{JS05}. In the present study, we show a different kind of symmetry in the vanishing of Ext over CM local rings with canonical modules, see Theorem~\ref{thm:Vanishing of Ext and Tor}. This considerably strengthens a result of Huneke-Jorgensen \cite[Thm.~2.1]{HJ03}, which was proved under the condition that $R$ is Gorenstein, and $M$, $N$ are MCM (maximal Cohen-Macaulay) $R$-modules.

\begin{customtheorem}{\ref{thm:Vanishing of Ext and Tor}}
	Let $R$ be a CM local ring with a canonical module $\omega$. Let $M$ and $N$ be CM $R$-modules. Set $M^{+} := \Ext_R^n(M,\omega)$, where $n=\dim(R)-\dim(M)$. Then the following are equivalent: {\rm (1)} $\Ext_R^{\gg 0}(M,N^{+}) = 0$, {\rm (2)} $\Ext_R^{\gg 0}(N,M^{+}) = 0$, and {\rm (3)} $\Tor_{\gg 0}^R(M,N) = 0$.
\end{customtheorem}

In the same direction, we also investigate symmetry in the vanishing of Tate cohomology over Gorenstein local rings.

\begin{customtheorem}{\ref{thm:Vanishing-Tate-coh}}
	Along with the hypotheses as in {\rm\ref{thm:Vanishing of Ext and Tor}}, further assume that $R$ is Gorenstein. Then the following are equivalent:\\
	{\rm (1)} $\widehat{\Ext}_R^{\ll 0}(M,N^{+}) = 0$, {\rm (2)} $\widehat{\Ext}_R^{\ll 0}(N,M^{+}) = 0$, and {\rm (3)} $\widehat{\Tor}_{\ll 0}^R(M, N) = 0$.
\end{customtheorem}

\section{A criterion for Gorenstein rings via G-dim-test modules}\label{sec:G-dim-test}

\begin{definition}\label{defn:G-dim-test} 
	An $R$-module $M$ is said to be a $\pd$-test (resp., $\gdim$-test) module provided that the following condition holds: Whenever $N$ is an $R$-module such that $\Tor_i^R(M,N)=0$ for all $i\gg 0$, one has that $\pd_R(N)<\infty$ (resp., $\gdim_R(N)<\infty$).
\end{definition}

If $M$ is a $\pd$-test module, then it is a $\gdim$-test module, but the converse is not necessarily true. There are $\gdim$-test modules that are not $\pd$-test modules; see, for example, \cite[3.11]{CS16}. As a class of $\pd$-test modules, we should note that if $I$ is an integrally closed ideal of $R$ such that $\depth(R/I) = 0$ (e.g., when $I$ is $\fm$-primary), then $I$ is a $\pd$-test (hence $\gdim$-test) $R$-module, \cite[2.3]{CS16}.

\begin{lemma}\label{modout}
	Let $M$ be a $\gdim$-test $R$-module. Let $x$ be a regular element over $R$ and $M$. Then $M/xM$ is a $\gdim$-test module over $R/xR$.
\end{lemma}

\begin{proof}
	Let $T$ be a module over $R/xR$. Let $ \Tor_i^{R/xR}(M/xM, T) = 0 $ for all $ i \gg 0 $. As $M$ is a $\gdim$-test $R$-module and $\Tor_i^R(M, T) \cong \Tor_i^{R/xR}(M/xM, T) = 0$ for all $i\gg 0$, we get that $\gdim_R(T)<\infty$. This implies $\gdim_{R/xR}(T)<\infty$; see, for example, \cite[page 39]{Cr}. Therefore, $M/xM$ is a $\gdim$-test module over $R/xR$.
\end{proof}

The following argument provides a different proof of \cite[4.4]{CS16} in the case when the ring is CM. Our proof is very short and elementary compare to that of \cite[4.4]{CS16}.

\begin{theorem}\label{thm:G-dim-test-Gor-charac}
	Let $R$ be a CM local ring, and let $M$ be a $\gdim$-test module over $R$. Assume that $ \Ext^i_R(M,R) = 0 $ for all $ i \gg 0 $. Then $R$ is Gorenstein.
\end{theorem}

\begin{proof} 
	Set $d:=\depth(R)$. First assume $d>0$. Replacing $M$ with a high syzygy, we may assume that $M$ is MCM. Pick an element $x \in R$ that is regular on both $R$ and $M$. Then it follows that $\Ext^{i}_{R/xR}(M/xM, R/xR)=0$ for all $i\gg 0$. As $R/xR$ is CM, and $M/xM$ is a $\gdim$-test $R/xR$-module (Lemma~\ref{modout}), we may proceed and reduce to the case where $R$ is Artinian. So assume $d=0$, i.e., $R$ is Artinian.
	
	Set $(-)^{\vee}:=\Hom_R(-,E)$, where $E$ is the injective hull of the residue field of $R$. Then $\Tor_i^R(M,E)^{\vee} \cong \Ext^i_R(M,E^{\vee}) \cong \Ext^i_R(M,R)$ for all $i\ge 0$. So it follows by our hypothesis that $\Tor_i^R(M,E)^{\vee}$, and hence $\Tor_i^R(M,E)$, vanishes for all $i\gg 0$. So $\gdim_R(E)<\infty$. This implies that $R$ is Gorenstein; see, e.g., \cite[1.2]{JL}.
\end{proof}

\section{Criteria for Gorenstein rings via G-dimension}\label{sec:charac-Gor-via-G-dim}

\begin{para}
	For an $R$-module $ M $ of depth $t$, the type of $ M $, denoted by $ \type(M) $, is defined to be the $t$th Bass number $ \dim_k\left( \Ext_R^t(k,M) \right) $.	The minimal number of generators of $M$ is denoted by $\mu(M)$, i.e., $\mu(M) := \dim_k(M\otimes_R k)$. Set $M^* := \Hom_R(M,R)$. We denote $ M^{\dagger} := \Ext_R^g(M,R) $, when $g=\gdim_R(M) < \infty$.
\end{para}

\begin{definition}\label{defn:grade-and-G-perfect}
	\begin{enumerate}[\rm (1)]
		\item The grade of $M$ is defined to be
		\begin{center}
			$\grade(M) = \min\{ i : \Ext_R^i(M,R) \neq 0 \}$.
		\end{center}
		\item \cite[(4.34)]{AB69} A nonzero $R$-module $M$ is called G-perfect if
		\begin{center}
			$ \gdim_R(M) = \grade(M)$,
		\end{center}
		equivalently, if
		\[
			\Ext_R^n(M,R) = 0 \mbox{ for all } n\neq g, \;\mbox{ where } g = \gdim_R(M) < \infty.
		\]
	\end{enumerate}
\end{definition}

A nonzero $R$-module $M$ is called perfect if $\pd_R(M) = \grade(M)$. Every perfect module is G-perfect, but the converse is not true in general. Indeed, every nonzero module of $G$-dimension $0$ is always G-perfect, which need not be perfect.

If $R$ is CM, and $M$ is an $R$-module of finite projective dimension, then it is well known that $M$ is CM if and only if $M$ is perfect, see \cite[2.1.5]{BH93}. This result is extended further to $G$-perfect modules by Golod \cite[p.63, 9.]{Go84}. We show that the `only if' part is true over any local ring. Note  that the `if' part is not true in general. For example, $M=R$ is always a perfect module, which need not be CM.

\begin{theorem}\label{thm:CM-iff-G-perfect-for-finite-G-dim}
	Let $M$ be a nonzero $R$-module of finite G-dimension. Then
	\begin{center}
		$M$ is CM \quad $ \Longrightarrow $ \quad $M$ is G-perfect.
	\end{center}
	The converse holds true when $R$ is CM.
\end{theorem}

\begin{proof}
	Let $M$ be a CM $R$-module of finite G-dimension. Set $d := \depth(R)$, $s := \dim(M)$, $g := \gdim_R(M)$ and $ h := \grade(M) $. By the Auslander-Bridger Formula \cite[(4.13).(b)]{AB69}, $g =\gdim_R(M) = \depth(R) - \depth(M) = d-s$ as $M$ is CM. In order to prove that $M$ is G-perfect, i.e., $g=h$, we use induction on $s$.
	
	In the base case, assume that $s=0$. Then $M$ has finite length, and $\Supp(M) = \{\fm\}$. It follows that $\depth(R) = \min\{ i : \Ext_R^i(M,R) \neq 0 \} = \grade(M)$, see, e.g., \cite[1.2.10.(e)]{BH93}. Hence $g=d=h$.
	
	Next assume that $s\ge 1$. Since $M$ is CM, there exists an $M$-regular element $x$. Then $M/xM$ is a CM $R$-module of dimension $s-1$. Moreover, $\gdim_R(M/xM) = \gdim_R(M) + 1 = g+1$ is finite. Therefore, by the induction hypothesis, $M/xM$ is a G-perfect $R$-module of G-dimension $g+1$. Hence
	\begin{equation}\label{Ext-M/xM-zero}
	\Ext_R^i(M/xM, R) = 0 \; \mbox{ for all } i \neq g+1.
	\end{equation}
	Considering the short exact sequence $0 \to M \stackrel{x}{\rightarrow} M \rightarrow M/xM \to 0$, we have an exact sequence
	\begin{equation}\label{exact-seq-on-Ext}
	\Ext_R^i(M, R) \stackrel{x}{\longrightarrow} \Ext_R^i(M, R) \longrightarrow \Ext_R^{i+1}(M/xM, R)\; \mbox{ for every } i.
	\end{equation}
	So it follows from \eqref{Ext-M/xM-zero}, \eqref{exact-seq-on-Ext} and the Nakayama's Lemma that $\Ext_R^i(M, R) = 0$ for all $ i \neq g$. Thus $M$ is G-perfect.
	
	The converse part can be proved similarly as in the case of perfect modules, see the proof of \cite[2.1.5]{BH93}.  Since $M$ is G-perfect and $R$ is CM, $\grade(M) = \gdim_R(M) = \dim(R) - \depth(M) $. On the other hand,
	\begin{center}
		$\grade(M) = \depth(\ann_R(M),R) = \dim(R) - \dim(M)$,
	\end{center}
	see, e.g., \cite[1.2.10.(e) and 2.1.4]{BH93}. Hence $\depth(M) = \dim(M)$, i.e., $M$ is CM.
\end{proof}

\begin{remark}
	In \cite{Fo75}, Foxby showed that if $R$ is CM, then an $R$-module $M$ of grade $g$ is G-perfect \iff $M$ is CM and $\Ext_R^g(M,R)$ is G-perfect. Avramov and Martsinkovsky in \cite[6.3]{AM02} proved that if an $R$-module $M$ is G-perfect of grade $g$, then $\Ext_R^g(M,R)$ is also so. The main implication in Theorem~\ref{thm:CM-iff-G-perfect-for-finite-G-dim} does not follow from these results. Rather it is a variation of Foxby's result.
\end{remark}
%

The following theorem shows that the type of a G-perfect $R$-module is a multiple of the type of $R$.

\begin{theorem}\label{thm:G-perfect-and-type-formula}
	If $M$ is a G-perfect $R$-module of G-dimension $g$, then
	\[
		\type(M) = \mu(\Ext_R^g(M,R)) \type(R).
	\]
\end{theorem}

In order to prove Theorem~\ref{thm:G-perfect-and-type-formula}, we use spectral sequences.

\begin{para}\label{spec-seq}
	Let $L$, $M$ and $N$ be $R$-modules. Let $\mathbb{P}_L$ be a projective resolution of $L$, and $\mathbb{I}_N$ be an injective resolution of $N$. The tensor-hom adjunction yields an isomorphism between the double complexes:
	\begin{equation}\label{eqn:double-complexes-X-Y}
	\mathbb{X} := \Hom_R(\mathbb{P}_L, \Hom_R(M,\mathbb{I}_N)) \cong \Hom_R(\mathbb{P}_L \otimes_R M, \mathbb{I}_N) =: \mathbb{Y}.
	\end{equation}
	Since $\Hom_R(\mathbb{P}_L, - )$ is an exact functor, computing cohomology of $\mathbb{X}$, first  vertically, and then horizontally, we obtain the first two pages of the spectral sequence
	\begin{align}
	{}^vE_1^{p,q}(\mathbb{X}) &= \Hom_R\left(P_p, \Ext_R^q(M,N)\right) \quad \mbox{and} \nonumber\\ {}^vE_2^{p,q}(\mathbb{X}) &= \Ext_R^p\left(L, \Ext_R^q(M,N)\right) \mbox{ for all } p,q.\label{spec-seq-X}
	\end{align}
	On the other hand, since $\Hom_R( - , \mathbb{I}_N)$ is an exact functor, computing cohomology of $\mathbb{Y}$, first horizontally, and then vertically, we get another spectral sequence
	\begin{align}
	{}^hE_1^{p,q}(\mathbb{Y}) &= \Hom_R\left( \Tor_p^R(L,M), I^q \right) \quad \mbox{and} \nonumber\\ {}^hE_2^{p,q}(\mathbb{Y}) &= \Ext_R^q\left( \Tor_p^R(L,M), N \right) \mbox{ for all } p,q.  \label{spec-seq-Y}
	\end{align}
	Since $\mathbb{X}\cong \mathbb{Y}$, their total complexes are isomorphic, and hence their cohomologies are also so, i.e.,
	\begin{equation}\label{coh-Tot-X-Y-iso}
	H^n(\Tot(\mathbb{X})) \cong H^n(\Tot(\mathbb{Y})) \; \mbox{ for all } n.
	\end{equation}
\end{para}

Now we are in a position to prove Theorem~\ref{thm:G-perfect-and-type-formula}.

\begin{proof}[Proof of Theorem~\ref{thm:G-perfect-and-type-formula}]
	Consider the spectral sequences discussed in \ref{spec-seq} for $L=k$ and $N=R$. Denote $\tau(-) := \type(-)$. Let $g = \gdim_R(M)$. Since $M$ is G-perfect, $g$ is finite and $\Ext_R^q(M,R) = 0$  for all $q \neq g$. So the spectral sequence \eqref{spec-seq-X}
	\[
		^vE_2^{p,q}(\mathbb{X}) = \Ext_R^p\left(k, \Ext_R^q(M,R)\right)
	\]
	collapses on the line $q=g$ at the second page. Hence
	\begin{equation*}
		H^n(\Tot(\mathbb{X})) \cong {}^vE_2^{n-g,g}(\mathbb{X}) = \Ext_R^{n-g}\left(k, \Ext_R^g(M,R)\right) \; \mbox{ for all } n,
	\end{equation*}
	see, e.g., \cite[5.2.7]{We94}. Set $s := \depth(M^{\dagger})$. Denoting $ M^{\dagger} = \Ext_R^g(M,R) $,
	\begin{align}\label{coh-Tot-X}
	&H^{s+g}(\Tot(\mathbb{X})) \cong \Ext_R^s(k, M^{\dagger} ) \cong k^{\tau(M^{\dagger})} \neq 0 \quad \mbox{and}\\
	&H^n(\Tot(\mathbb{X})) = 0 \;\mbox{ for all } n < s+g.\nonumber
	\end{align}
	Set $d:=\depth(R)$. Since $\Tor_p^R(k,M)$ is a free $k$-module, the other spectral sequence \eqref{spec-seq-Y} satisfies that
	\begin{equation*}
	{}^hE_2^{p,q}(\mathbb{Y}) = \Ext_R^q\left( \Tor_p^R(k,M), R \right)= 0 \; \mbox{ for all } p <0 \,\mbox{ or }\, q < d.
	\end{equation*}
	Thus, for $p<0$ or $q<d$, ${}^hE_r^{p,q}=0$ for all $r \ge 2$, which imply that ${}^hE_{\infty}^{p,q}=0$. Moreover, since the differentials in the $r$th page are given by ${}^h d_r^{p,q} :{}^hE_r^{p,q} \longrightarrow {}^hE_r^{p-r+1,q+r}$, it can be noticed that
	\begin{center}
		$ {}^h d_r^{r-1,d-r}  : {}^h E_r^{r-1,d-r} \longrightarrow {}^h E_r^{0,d}$ \; and \; ${}^h d_r^{0,d} : {}^h E_r^{0,d} \longrightarrow {}^h E_r^{-r+1, d+r} $
	\end{center}
	are the zero maps for all $r\ge 2$, which yield that ${}^hE_{\infty}^{0,d} = {}^hE_2^{0,d}$. Thus  ${}^hE_{\infty}^{0,d}$ is the only possible nonzero term on the line $p+q=d$ in the abutment of the spectral sequence ${}^hE_r^{p,q}(\mathbb{Y})$, while all the terms on the line $p+q=n$ for $n<d$ are zero. So, in view of the filtration of $H^n(\Tot(\mathbb{Y}))$ (cf.~\cite[5.2.5 and 5.2.6]{We94}), one obtains that
	\begin{align}\label{coh-Tot-Y}
	H^d(\Tot(\mathbb{Y})) & \cong {}^hE_{\infty}^{0,d}(\mathbb{Y}) = {}^hE_2^{0,d}(\mathbb{Y}) \\
	& = \Ext_R^d \left( \Tor_0^R(k,M), R \right) \cong k^{\mu(M)\tau(R)} \neq 0 \quad \mbox{and} \nonumber\\
	H^n(\Tot(\mathbb{Y})) & = 0 \; \mbox{ for all } \; n<d.\nonumber
	\end{align}
	It follows from \eqref{coh-Tot-X-Y-iso}, \eqref{coh-Tot-X} and \eqref{coh-Tot-Y} that for a G-perfect $R$-module $M$, 
	\begin{equation}\label{equality-type-mu}
	\tau(M^{\dagger}) = \mu(M)\tau(R).
	\end{equation}
	By \cite[Thm.~6.3]{AM02}, since $M^{\dagger}$ is also a G-perfect $R$-module (of grade $g$) and $M \cong M^{\dagger\dagger}$, applying \eqref{equality-type-mu} for $M^{\dagger}$, one derives that $\tau(M) = \tau(M^{\dagger\dagger}) = \mu(M^{\dagger})\tau(R)$, which completes the proof.
\end{proof}

A module which is not G-perfect may also satisfy the equality in Theorem~\ref{thm:G-perfect-and-type-formula}.

\begin{example}\label{exam:Hom-m-omega}
	Let $(R,\fm,k)$ be a Gorenstein local ring of dimension $d \ge 2$. Then the exact sequence $0\to \fm \to R \to k \to 0$ induces that $\Hom_R(\fm,R) \cong R$ and
	\[
	\Ext_R^i(\fm,R) \cong \Ext_R^{i+1}(k,R) \cong \left\{\begin{array}{ll}
	0 & \mbox{if } i \ge 1 \mbox{ and } i \neq d-1,\\
	k \neq 0 & \mbox{if } i = d-1.
	\end{array}\right.
	\]
	Therefore $\fm$ is not a G-perfect $R$-module, $\gdim_R(\fm) = d-1$ and $\depth(\fm) = 1$. Since $\Ext^1_R(k,\fm) \cong \Hom_R(k,k) \cong k $, $\type(\fm) = 1 = \mu(\Ext^{d-1}_R(\fm,R))\type(R)$.
\end{example}

There are examples of CM (non-free) $R$-modules $M$ of finite G-dimension $g$ such that $\type(M) = \mu( \Ext_R^g(M,R) )$, where $\type(M)$ can be arbitrarily large.

\begin{example}\label{exam:G-perfect-non-free-modules}
	Let $R = k[X_1,\ldots,X_n]/\langle X_1^2,\ldots,X_n^2 \rangle$ over a field $k$, where $n \ge 2$. Clearly, $R=k[x_1,\ldots,x_n]$ is a Gorenstein Artinian local ring with the maximal ideal $\fm = \langle x_1,\ldots,x_n\rangle$. So every $R$-module is CM and of G-dimension $0$ (hence G-perfect). Set $I := \langle x_i x_j : 1 \le i < j \le n \rangle$ and $M := R/I$. Note that $M$ is non-free, and $\Soc(M) =  (0 :_M \fm) = kx_1\oplus \cdots \oplus k x_n$. Thus $\type(M) = \dim_k(\Soc(M)) = n$. On the other hand, $\Hom_R(M,R) \cong (0 :_R I)$ is given by
	\begin{equation}
	K\mbox{-Span}\{ x_1 x_2\cdots x_n, \;\, x_1 \cdots x_{i-1} \widehat{x_i} x_{i+1} \cdots x_n : 1 \le i \le n \}.
	\end{equation}
	So $\mu(\Ext_R^0(M,R)) = n$. Since $\type(R) = 1$, the equality
	\begin{center}
		$\type(M) = \mu(\Ext_R^0(M,R)) \type(R)$
	\end{center}
	holds true as in Theorem~\ref{thm:G-perfect-and-type-formula}.
\end{example}

In view of Theorem~\ref{thm:G-perfect-and-type-formula} and Example~\ref{exam:Hom-m-omega}, the following natural question arises.

\begin{question}
	Does there exist an $R$-module $M$ of finite G-dimension $g$ such that $\type(M) \neq \mu(\Ext_R^g(M,R)) \type(R)$?
\end{question}

As consequences of Theorems~\ref{thm:CM-iff-G-perfect-for-finite-G-dim} and \ref{thm:G-perfect-and-type-formula}, we obtain the following characterizations of Gorenstein local rings in terms of the existence of certain modules of finite G-dimension.

\begin{corollary}\label{cor:Gor-iff-there-is-CM-G-dim-finite-mod}
	The  following statements are equivalent:
	\begin{enumerate}[\rm (1)]
		\item $R$ is Gorenstein.
		\item $R$ admits a nonzero CM module $M$ of finite G-dimension $g$ such that
		\begin{center}
			$\type(M) \le \mu( \Ext_R^g(M,R) )$ \quad {\rm (}e.g., $\type(M)=1${\rm )}.
		\end{center}
		\item $R$ is CM, and it admits a G-perfect module $M$ of G-dimension $g$ such that $\type(M) \le \mu( \Ext_R^g(M,R) )$.
	\end{enumerate}
\end{corollary}

\begin{proof}
	(1) $\Rightarrow$ (2) and (1) $\Rightarrow$ (3): Since $R$ is Gorenstein, it is CM of type $1$, see, e.g., \cite[3.2.10]{BH93}. Note that $M = R$ satisfies the conditions (2) and (3).
	
	(2) $\Rightarrow$ (1):
	By virtue of Theorem~\ref{thm:CM-iff-G-perfect-for-finite-G-dim}, $M$ is G-perfect. So Theorem~\ref{thm:G-perfect-and-type-formula} yields that $\type(M) = \mu(M^{\dagger}) \type(R)$. Since $\type(M) \le \mu(M^{\dagger})$, it follows that $R$ has type $1$. Therefore, by \cite[Thm.~2.3.(i)$\Leftrightarrow$(iv)]{Ta04}, $R$ is Gorenstein.
	
	(3) $\Rightarrow$ (1):
	From the given inequality, by Theorem~\ref{thm:G-perfect-and-type-formula}, it follows that $\type(R)=1$. Therefore, since $R$ is CM, $R$ must be Gorenstein, see, e.g., \cite[3.2.10]{BH93}.
\end{proof}

\begin{remark}\label{rmk:strengthening-Ryo}
	There exist CM (non-free) $R$-modules $M$ of finite G-dimension $g$ such that $\type(M) = \mu( \Ext_R^g(M,R) )$, where $\type(M)$ can be arbitrarily large, see, e.g., Example~\ref{exam:G-perfect-non-free-modules}). It justifies that Corollary~\ref{cor:Gor-iff-there-is-CM-G-dim-finite-mod}.(1)$\Leftrightarrow$(2) considerably strengthens \cite[Thm.~2.3.(i)$\Leftrightarrow$(iii)]{Ta04}.
\end{remark}

\begin{remark}
	The existence of a G-perfect $R$-module of type $1$ does not necessarily imply that $R$ is Gorenstein. For example, $R = K[[X,Y]]/(X^2,XY)$ has type $1$, and it is not CM. So $M=R$ is a G-perfect $R$-module of type $1$, but $R$ is not Gorenstein. So the conditions `$M$ is CM' in (2) and `$R$ is CM' in (3) cannot be removed.
\end{remark}

\begin{remark}
	The existence of a nonzero CM $R$-module $M$ of finite G-dimension does not guarantee that $R$ is Gorenstein, for example $M=R$, where $R$ is a non-Gorenstein CM local ring. So the assumption that $\type(M) \le \mu( \Ext_R^g(M,R) )$ cannot be omitted from (2) and (3) in Corollary~\ref{cor:Gor-iff-there-is-CM-G-dim-finite-mod}.
\end{remark}

\section{When Ext modules have finite injective dimension}
	
	In this section, we study the consequences of certain Ext modules having finite injective dimension. We first generalize a result of Foxby, which was proved for CM modules in \cite[Prop.~3.5.(b)]{Fo71}, see Remark~\ref{rmk:Foxby}.
	
	\begin{proposition}\label{prop:injdim-Ext-finite}
		Let $R$ be a CM local ring with a canonical module $\omega$. Let $M$ be an $R$-module. If $\id_R( \Ext_R^i(M,\omega) )$ is finite for all $i \ge 0$, then $\pd_R(M)$ is finite.
	\end{proposition}
	
	\begin{proof}
		Let $\id_R( \Ext_R^i(M,\omega) ) < \infty$ for all $i \ge 0$. In order to prove that $\pd_R(M)$ is finite, we may assume that $M$ is MCM as $M$ can be replaced by the syzygy module $\Omega_R^d(M)$, which is MCM, where $d=\dim(R)$. To justify this reduction, it is enough to showing that $\id_R( \Ext_R^i(\Omega(M),\omega) )$ is finite for all $i$.  Considering the short exact sequence $0 \to \Omega(M) \to R^b  \to M \to 0$, we get an exact sequence:
		\begin{align}\label{exact-seq-Ext-injdim}
		0 \longrightarrow & \Hom_R(M,\omega) \longrightarrow \omega^b \longrightarrow \Hom_R(\Omega(M),\omega) \longrightarrow \\
		&\Ext^1_R(M,\omega) \longrightarrow 0 \longrightarrow \Ext^1_R(\Omega(M),\omega) \longrightarrow \nonumber \\
		&\Ext^2_R(M,\omega) \longrightarrow 0 \longrightarrow \Ext^2_R(\Omega(M),\omega) \longrightarrow \cdots.\nonumber
		\end{align}
		Since $\omega$ and $\Ext_R^i(M,\omega)$ for all $i \ge 0$ have finite injective dimensions, it follows from the exact sequence \eqref{exact-seq-Ext-injdim} that $\id_R( \Ext_R^i(\Omega(M),\omega) )$ is finite for all $i$.
		
		Since $M$ is MCM, $\Ext_R^i(M,\omega) = 0$ for all $i \neq 0$, and $\Hom_R(M,\omega)$ is also MCM (see, e.g., \cite[3.3.3]{BH93}). Thus $\Hom_R(M,\omega)$ is MCM of finite injective dimension. Therefore, by \cite[3.3.28]{BH93}, $\Hom_R(M,\omega) \cong \omega^r$ for some $r\ge 1$. Hence, applying the $\omega$-dual, one obtains that $M$ is free (cf.~\cite[3.3.10]{BH93}).
	\end{proof}
	
	\begin{remark}\label{para:inf-sup-nonvanishing}
		Let $M$ and $N$ be nonzero $R$-modules. Then
		\begin{center}
			$\depth(\ann_R(M),N) = \min\{ i : \Ext_R^i(M,N) \neq 0 \}$.
		\end{center}
		If $\id_R(N) < \infty$, then we also have that
		\begin{center}
			$\dim(R) - \depth(M) = \sup\{ i : \Ext_R^i(M,N) \neq 0 \}$,
		\end{center}
		see, e.g., \cite[1.2.10.(e) and 3.1.24]{BH93}.
	\end{remark}

	\begin{remark}\label{rmk:Foxby}
		Let $R$ be a CM local ring of dimension $d$ with a canonical module $\omega$. Let $M$ be a CM $R$-module of dimension $s$. Then, by Remark~\ref{para:inf-sup-nonvanishing}, $\Ext_R^i(M,\omega)\neq 0$ \iff $i = d-s$. In \cite[Prop.~3.5.(b)]{Fo71}, Foxby showed that $\id_R(\Ext_R^{d-s}(M,\omega))$ is finite \iff $ \pd_R(M) $ is so. Note that Proposition~\ref{prop:injdim-Ext-finite} generalizes the `only if' part, while the `if' part is trivial. Moreover, our proof is elementary, and it does not use spectral sequence.
	\end{remark}

	\begin{remark}
		We got to know from Holanda that a similar result as in Proposition~\ref{prop:injdim-Ext-finite} is proved in \cite{TH} by Thiago-Holanda.
		Let $R$ be a homomorphic image of a Gorenstein local ring $S$.
		Let $ M $ be an $ R $-module
		such that $ \id_R(\Ext_S^i(M,S)) < \infty $ for all $i$.
		Then $ \pd_R(M) < \infty $, see \cite[Cor.~4.7]{TH}. When $R$ is CM, then Proposition~\ref{prop:injdim-Ext-finite} is same as \cite[Cor.~4.7]{TH}, due to local duality, see, e.g., \cite[11.2.6 and 12.1.20]{BS98}. However, unlike Proposition~\ref{prop:injdim-Ext-finite}, the proof of \cite[Cor.~4.7]{TH} uses spectral sequence arguments.
	\end{remark}

	In view of Proposition~\ref{prop:injdim-Ext-finite}, we are mainly interested in the following question.
	
	\begin{question}\label{ques:injdim-Ext-finite}
		Let $M$ and $N$ be nonzero $R$-modules. Suppose that $\id_R(N) < \infty$. If $\id_R( \Ext_R^i(M,N) )$ is finite for all $i \ge 0$, then is $\pd_R(M)$ finite?
	\end{question}
	
	When $N$ is MCM and $\id_R(N) < \infty$, then $N \cong \omega^r$ for some $r \ge 1$, see, e.g., \cite[3.3.28]{BH93}. Hence, by Proposition~\ref{prop:injdim-Ext-finite}, Question~\ref{ques:injdim-Ext-finite} has a positive answer if $N$ is MCM. In the following theorem, we show that this question also has an affirmative answer when $\depth(N) = \dim(R)-1$.
	
	\begin{theorem}\label{thm:injdim-Ext-finite-N-d-1}
		Let $M$ and $N$ be nonzero $R$-modules. Suppose that $\id_R(N) < \infty$ and $\depth(N) \ge \dim(R)-1$. If $\id_R( \Ext_R^i(M,N) )$ is finite for all $i \ge 0 $, then $\pd_R(M)$ is finite.
	\end{theorem}
	
	\begin{proof}
		Let $\id_R( \Ext_R^i(M,N) ) < \infty$ for all $i\ge 0$. As in the proof of Proposition~\ref{prop:injdim-Ext-finite}, we may assume that $M$ is MCM. Then, by a result of Ischebeck (cf.~\cite[3.1.24]{BH93}), $\Ext_R^i(M,N) = 0$ for all $i \neq 0$. We consider the spectral sequences discussed in \ref{spec-seq} for $L=k$. It follows that the spectral sequence ${}^vE_2^{p,q}(\mathbb{X}) = \Ext_R^p\left(k, \Ext_R^q(M,N)\right)$ collapses on the line $q=0$. Hence the cohomologies of the total complex of $\mathbb{X}$ are given by
		\begin{equation}\label{coh-Tot-X-id}
		H^n(\Tot(\mathbb{X})) \cong {}^vE_2^{n,0}(\mathbb{X}) = \Ext_R^n(k, \Hom_R(M,N)) \; \mbox{ for all } n,
		\end{equation}
		see, e.g., \cite[5.2.7]{We94}. Set $d:=\dim(R)$. Since $\Hom_R(M,N)$ has finite injective dimension, $\id_R( \Hom_R(M,N) ) = d$, and hence
		\begin{equation}\label{coh-Tot-X-zero}
			H^n(\Tot(\mathbb{X})) =0 \; \mbox{ for all } n > d.
		\end{equation}
		On the other hand, since $\Tor_p^R(k,M)$ is a free $k$-module for every $p$, and $\depth(N) \ge d-1$, the other spectral sequence has at most two nonzero rows, namely $q=d-1$ and $q=d$, i.e., 
		\begin{equation}\label{spec-seq-Y-zero-id}
		{}^hE_2^{p,q}(\mathbb{Y}) = \Ext_R^q\left( \Tor_p^R(k,M), N \right) = 0 \; \mbox{ for all } q \neq d-1, \,d.
		\end{equation}
		Note that the differentials in ${}^hE_2^{p,q}$ are given by ${}^h d_2^{p,q} :{}^hE_2^{p,q} \longrightarrow {}^hE_2^{p-1,q+2}$. So \eqref{spec-seq-Y-zero-id} yields that ${}^hE_{\infty}^{p,q}(\mathbb{Y}) = {}^hE_2^{p,q}(\mathbb{Y})$. Therefore there is a short exact sequence:
		\begin{align}
		0 \longrightarrow {}^hE_2^{d+1,d}(\mathbb{Y}) \longrightarrow H^{2d+1}(\Tot(\mathbb{Y})) \longrightarrow {}^hE_2^{d+2,d-1}(\mathbb{Y}) \longrightarrow 0,
		\end{align}
		see, e.g., \cite[10.29]{Ro09}. Since $H^{2d+1}(\Tot(\mathbb{Y})) \cong H^{2d+1}(\Tot(\mathbb{X})) = 0$ (by \eqref{coh-Tot-X-zero}), it follows that
		\begin{center}
			${}^hE_2^{d+1,d}(\mathbb{Y}) = \Ext_R^d\left( \Tor_{d+1}^R(k,M), N \right) = 0$,
		\end{center}
		which implies that $\Tor_{d+1}^R(k,M) = 0$ (see \cite[3.1.24]{BH93}). Thus $\pd_R(M)$ is finite.
	\end{proof}
	
	As an application of Theorem~\ref{thm:injdim-Ext-finite-N-d-1}, we obtain the following characterization of Gorenstein local rings in terms of injective dimension of certain modules.
	
	\begin{corollary}\label{cor:injdim-Ext-finite-N-d-1-charac-Gor}
		Suppose there is a nonzero $R$-module $M$ of depth $\ge \dim(R) - 1$ such that the injective dimensions of $M$, $\Hom_R(M,M)$ and $\Ext_R^1(M,M)$ are finite. Then, $M$ has finite projective dimension, and $R$ is Gorenstein.
	\end{corollary}
	
	\begin{proof}
		In view of Remark~\ref{para:inf-sup-nonvanishing}, $\Ext_R^i(M,M) = 0$ for all $i > 1$. So, it follows from Theorem~\ref{thm:injdim-Ext-finite-N-d-1} that $\pd_R(M)$ is finite. Thus both the projective and injective dimensions of $M$ are finite. Hence, by Theorem~\ref{thm:known-charac-via-inj-dim}.(2), $R$ is Gorenstein.
	\end{proof}
	
%
	
	\begin{remark}
		In \cite{GT21}, a few criteria are provided for a module to be free and for a local ring to be Gorenstein in terms of vanishing of certain Ext and injective dimension of Hom. Along with some partial positive answers, a question was raised in \cite[Ques.~2.18]{GT21} whether $R$ is Gorenstein if there exists a nonzero $R$-module $M$ such that $\Hom_R(M,M)$ has finite injective dimension. Corollary~\ref{cor:injdim-Ext-finite-N-d-1-charac-Gor} provides another partial positive answer to this question.
	\end{remark}

	In view of Corollary~\ref{cor:injdim-Ext-finite-N-d-1-charac-Gor}, one may ask whether $ M $ and $ \Ext_R^i(M,M) $ having finite injective dimensions for all $0 \le i \le \dim(R) - \depth(M)$ imply that $M$ has finite projective dimension. It is equivalent to the following question.
%
%
	
	\begin{question}\label{ques:inj-dim-Ext-finite-is-R-Gor}
		Let $M$ be a nonzero $R$-module such that the injective dimensions of $M$ and $\Ext_R^i(M,M)$ are finite for all $0 \le i \le \dim(R) - \depth(M)$. Is $R$ Gorenstein?
	\end{question}
	
	Note that Question~\ref{ques:inj-dim-Ext-finite-is-R-Gor} has a positive answer when $\dim(R) - \depth(M) \le 1$ as shown in Corollary~\ref{cor:injdim-Ext-finite-N-d-1-charac-Gor}.

\section{Symmetry in vanishing of Ext and Tate cohomology}

Throughout this section, we work with the following setup.

\begin{setup}\label{setup:CM-ring-M+}
	Let $(R,\fm)$ be a CM local ring of dimension $d$ with a canonical module $\omega$. Let $M$ and $N$ be $R$-modules. Set $M^{+} := \Ext_R^{d-\dim(M)}(M,\omega)$.
\end{setup}

With Setup~\ref{setup:CM-ring-M+}, we prove the following result on the vanishing of Ext and Tor. This considerably strengthens a result of Huneke-Jorgensen \cite[Thm.~2.1]{HJ03}, which was proved under the condition that $R$ is a Gorenstein local ring, and $M$, $N$ are MCM $R$-modules.

\begin{theorem}\label{thm:Vanishing of Ext and Tor}
	Let $M$ and $N$ be CM $R$-modules. Then the following are equivalent:
	\begin{enumerate}[\rm (1)]
		\item $\Ext_R^i(M,N^{+}) = 0$ for all $i \gg 0$.
		\item $\Ext_R^i(N,M^{+}) = 0$ for all $i \gg 0$.
		\item $\Tor_i^R(M,N) = 0$ for all $i \gg 0$.
	\end{enumerate}
\end{theorem}

\begin{proof}
	The equivalences in the theorem are obtained from Lemma~\ref{lem:vanshing-ext-tor} and the fact that $\Tor_i^R(M, N) \cong \Tor_i^R(N, M)$ for all $i$.
\end{proof}

The following facts are well known, which are used to prove Lemma~\ref{lem:vanshing-ext-tor}.

\begin{proposition}\label{prop:properties-of-M+}
	Let $M$ be a CM $R$-module of dimension $t$. Then
	\begin{enumerate}[\rm (1)]
		\item $\Ext_R^i(M,\omega) = 0$ for all $i \neq d-t$.
		\item The module $M^{+} = \Ext_R^{d-t}(M,\omega)$ is also CM of dimension $t$, and $M^{++} \cong M$.
		\item If $x$ is an $M$-regular element, then $x$ is also $M^{+}$-regular. Moreover,
		\begin{center}
			$(M/xM)^{+} \cong M^{+}/xM^{+}$.
		\end{center}
	\end{enumerate}
\end{proposition}

\begin{proof}
	See \cite[3.3.10]{BH93} for the assertions (1) and (2). If $x$ is $M$-regular, then $M/xM$ is CM of dimension $t-1$. Therefore, by (1), the exact sequence $0 \to M \stackrel{x}{\to} M \to M/xM \to 0$ yields another exact sequence
	\begin{center}
		$0 \to \Ext_R^{d-t}(M,\omega) \stackrel{x}{\to} \Ext_R^{d-t}(M,\omega) \to \Ext_R^{d-(t-1)}(M/xM,\omega) \to 0$,
	\end{center}
	hence (3) follows.
\end{proof}

\begin{lemma}\label{lem:vanshing-ext-tor}
	Let $L$ and $M$ be $R$-modules, where $M$ is CM. Then
	\begin{center}
		$\Ext_R^i(L,M^{+}) = 0$ for all $i \gg 0$ \ $\Longleftrightarrow$ \ $\Tor_i^R(L, M) = 0$ for all $i \gg 0$.
	\end{center}
\end{lemma}

\begin{proof}
	We use induction on $t := \dim(M)$.	In the base case, assume that $t=0$. Consider the spectral sequences discussed in \ref{spec-seq} for $N=\omega$. In view of \ref{prop:properties-of-M+}.(1), $\Ext_R^q(M,\omega) = 0$ for all $q \neq d$. So the spectral sequence ${}^vE_2^{p,q}(\mathbb{X}) = \Ext_R^p\left(L, \Ext_R^q(M,\omega)\right)$	collapses on the line $q=d$. Hence, for all $n$, the cohomologies of the total complex of $\mathbb{X}$ are given by
	\begin{equation}\label{coh-Tot-X-sym}
	H^n(\Tot(\mathbb{X})) \cong {}^vE_2^{n-d,d}(\mathbb{X}) = \Ext_R^{n-d}(L, \Ext_R^d(M,\omega)) = \Ext_R^{n-d}(L, M^{+})
	\end{equation}
	see, e.g., \cite[5.2.7]{We94}. Since $\dim(M)=0$, $\Tor_p^R(L,M)$ has finite length for every $p$, and hence the other spectral sequence also collapses on the line $q=d$, i.e., 
	\begin{equation}\label{spec-seq-Y-zero-sym}
	{}^hE_2^{p,q}(\mathbb{Y}) = \Ext_R^q\left( \Tor_p^R(L,M), \omega \right) = 0 \; \mbox{for all} \; q \neq d.
	\end{equation}
	Thus the cohomologies of the total complex of $\mathbb{Y}$ are
	\begin{align}\label{coh-Tot-Y-sym}
	H^n(\Tot(\mathbb{Y})) &= {}^hE_2^{n-d,d}(\mathbb{Y}) = \Ext_R^d \left( \Tor_{n-d}^R(L,M), \omega \right) \; \mbox{for all} \; n.
	\end{align}
	Since $H^n(\Tot(\mathbb{X})) \cong H^n(\Tot(\mathbb{Y}))$ for all $n$, it follows from \eqref{coh-Tot-X-sym} and \eqref{coh-Tot-Y-sym} that
	\begin{align*}
		\Ext_R^i(L,M^{+}) = 0\; \mbox{for all} \; i \gg 0  &\Longleftrightarrow  \Ext_R^d \left( \Tor_i^R(L,M), \omega \right) = 0 \; \mbox{for all} \; i \gg 0.\\
		&\stackrel{(\dagger)}{\Longleftrightarrow} \Tor_i^R(L, M) = 0 \; \mbox{for all} \; i \gg 0,
	\end{align*}
	where $(\dagger)$ is obtained by Proposition~\ref{prop:properties-of-M+}.(2). This completes the base case.
	
	For the inductive step, assume that $t \ge 1$. Then there is an $M$-regular element $x$. In view of Proposition~\ref{prop:properties-of-M+}.(3), $x$ is also $M^{+}$-regular, and $(M/xM)^{+} \cong M^{+}/xM^{+}$. Applying $\Hom_R(L,-)$ on the short exact sequence $0 \to M^{+} \stackrel{x}{\longrightarrow} M^{+} \longrightarrow (M/xM)^{+} \to 0$, one obtains a long exact sequence of Ext modules. Hence, by the Nakayama's Lemma, it can be derived that
	\begin{equation}\label{Ext-equiv-mod-x}
	\Ext_R^i(L,M^{+}) = 0 \ \mbox{for all} \ i \gg 0 \Longleftrightarrow \Ext_R^i(L,(M/xM)^{+}) = 0 \ \mbox{for all} \ i \gg 0.
	\end{equation}
	Similarly, applying $L \otimes_R (-)$ on the short exact sequence $0 \to M \stackrel{x}{\to} M \to M/xM \to 0$, there is a long exact sequence of Tor modules, which implies that
	\begin{equation}\label{Tor-equiv-mod-x}
	\Tor_i^R(L, M) = 0 \; \mbox{for all} \; i \gg 0 \; \Longleftrightarrow \; \Tor_i^R(L, M/xM) = 0 \; \mbox{for all} \; i \gg 0.
	\end{equation}
	Since $M/xM$ is CM of dimension $t-1$, by the induction hypothesis,
	\begin{equation}\label{Ext-Tor-vanishing-equiv-induction-step}
	\Ext_R^i(L,(M/xM)^{+}) = 0 \ \mbox{for all} \ i \gg 0 \Leftrightarrow \Tor_i^R(L,M/xM) = 0 \ \mbox{for all} \ i \gg 0.
	\end{equation}
	Thus the desired equivalence follows from \eqref{Ext-equiv-mod-x}, \eqref{Tor-equiv-mod-x} and \eqref{Ext-Tor-vanishing-equiv-induction-step}.
\end{proof}

%

Before proving the result on symmetry in vanishing of Tate cohomologies, we first recall the notions of Tate (co)homologies, which are defined by using complete resolutions of modules.

\begin{para}\label{Tate-co-homology}
	A complete resolution of an $R$-module $M$ is defined to be a complex $T$ of projective $R$-modules satisfying the following: (i) the homologies $H_n(T) = 0$ and the cohomologies $H^n(\Hom_R(T,R)) = 0$ for all $n \in \bz$, and (ii) $T_{\ge r} = P_{\ge r}$ for some projective resolution $P$ of $M$ and some integer $r$. These are well known that $M$ has a complete resolution \iff $\gdim_R(M)<\infty$ (due to \cite{AM02}), and if $M$ has complete resolutions, then any two of them are homotopy equivalent, see, e.g., \cite[(2.4)]{CK97}. If $M$ has a complete resolution, and $N$ is an $R$-module, for all $i\in\bz$, the Tate (co)homologies of $M$ and $N$ are defined as
	\[
		\widehat{\Ext}_R^i(M,N) := H^i(\Hom_R(T,N)) \; \mbox{ and } \; \widehat{\Tor}^R_i(M,N) := H_i(T \otimes_R N)
	\]
\end{para}

For the rest of this section, we work over the following setup.

\begin{setup}\label{setup:Gor-ring-M+}
	Let $(R,\fm)$ be a Gorenstein local ring of dimension $d$. Let $L$, $M$ and $N$ be $R$-modules. Note that $M^{+} = \Ext_R^{d-\dim(M)}(M,R)$.
\end{setup}

The following lemma gives a relation between the vanishing of Tate cohomology and that of Tate homology.

\begin{lemma}\label{lem:vanshing-ext-tor-Tate}
	Along with Setup~\ref{setup:Gor-ring-M+}, suppose $M$ is CM. Then
	\begin{center}
		$\widehat{\Ext}_R^i(L,M^{+}) = 0$ for all $i \ll 0$ \ $\Longleftrightarrow$ \ $\widehat{\Tor}_i^R(L, M) = 0$ for all $i \ll 0$.
	\end{center}
\end{lemma}

\begin{proof}
	Here also we use induction on $t := \dim(M)$. First assume that $t=0$. Let $\mathbb{T}_L$ be a complete resolution of $L$, and $\mathbb{I}_R$ be a minimal injective resolution of $R$. We consider the spectral sequence as described in \ref{spec-seq} for $N=R$, and replacing $\mathbb{P}_L$ by $\mathbb{T}_L$. In view of Proposition~\ref{prop:properties-of-M+}.(1), $\Ext_R^q(M,R) = 0$ for all $q \neq d$. So the spectral sequence ${}^vE_2^{p,q}(\mathbb{X}) = \widehat{\Ext}_R^p\left(L, \Ext_R^q(M,R)\right)$ collapses on the line $q=d$. Hence, for all $n$, the cohomologies of the total complex of $\mathbb{X}$ are given by
	\begin{equation}\label{Tate-coh-Tot-X-sym}
	H^n(\Tot(\mathbb{X})) \cong {}^vE_2^{n-d,d}(\mathbb{X}) = \widehat{\Ext}_R^{n-d}\big(L, \Ext_R^d(M,R)\big) = \widehat{\Ext}_R^{n-d}(L, M^{+})
	\end{equation}
	see, e.g., \cite[5.2.7]{We94}. Since $\dim(M)=0$, $\widehat{\Tor}_p^R(L,M)$ has finite length for every $p$, and hence the other spectral sequence also collapses on the line $q=d$, i.e., 
	\begin{equation}\label{Tate-spec-seq-Y-zero-sym}
	{}^hE_2^{p,q}(\mathbb{Y}) = \Ext_R^q\left( \widehat{\Tor}_p^R(L,M), R \right) = 0 \; \mbox{for all} \; q \neq d.
	\end{equation}
	Thus the cohomologies of the total complex of $\mathbb{Y}$ are
	\begin{align}\label{Tate-coh-Tot-Y-sym}
	H^n(\Tot(\mathbb{Y})) &= {}^hE_2^{n-d,d}(\mathbb{Y}) = \Ext_R^d \big( \widehat{\Tor}_{n-d}^R(L,M), R \big) \; \mbox{for all} \; n.
	\end{align}
	Since $H^n(\Tot(\mathbb{X})) \cong H^n(\Tot(\mathbb{Y}))$ for all $n$, it follows from \eqref{Tate-coh-Tot-X-sym} and \eqref{Tate-coh-Tot-Y-sym} that
	\begin{align*}
	\widehat{\Ext}_R^i(L,M^{+}) = 0\; \mbox{for all} \; i \ll 0 \; &\Longleftrightarrow \; \Ext_R^d \big( \widehat{\Tor}_i^R(L,M), R \big) = 0 \; \mbox{for all} \; i \ll 0.\\
	&\Longleftrightarrow \; \widehat{\Tor}_i^R(L, M) = 0 \; \mbox{for all} \; i \ll 0 \; \mbox{[by \ref{prop:properties-of-M+}.(2)].}
	\end{align*}
	This completes the base case.
	
	For the inductive step, suppose $t \ge 1$. There is an $M$-regular element $x$. By \ref{prop:properties-of-M+}.(3), $x$ is also $M^{+}$-regular, and $(M/xM)^{+} \cong M^{+}/xM^{+}$. For the exact sequence $0 \to M^{+} \stackrel{x}{\longrightarrow} M^{+} \longrightarrow (M/xM)^{+} \to 0$, there is a two-sided long exact sequence of Tate cohomologies:
	\begin{align*}
	\cdots & \longrightarrow \widehat{\Ext}_R^i(L,M^{+}) \stackrel{x}{\longrightarrow} \widehat{\Ext}_R^i(L,M^{+}) \longrightarrow \widehat{\Ext}_R^i(L, (M/xM)^{+}) \\
	& \longrightarrow   \widehat{\Ext}_R^{i+1}(L,M^{+}) \longrightarrow \cdots.
	\end{align*}
	From this, by the Nakayama's Lemma, it can be derived that
	\begin{align}\label{Ext-equiv-mod-x-Tate}
	\widehat{\Ext}_R^i(L,M^{+}) = 0 \ \mbox{for all}\ i \ll 0 \Longleftrightarrow \widehat{\Ext}_R^i(L, (M/xM)^{+}) = 0  \ \mbox{for all}\ i \ll 0.
	\end{align}
	The exact sequence $0 \to M \stackrel{x}{\to} M \to M/xM \to 0$ yields a long exact sequence of Tate homologies:
	\[
		\cdots \rightarrow \widehat{\Tor}_i^R(L, M) \stackrel{x}{\longrightarrow} \widehat{\Tor}_i^R(L, M) \longrightarrow \widehat{\Tor}_i^R(L, M/xM) \longrightarrow \widehat{\Tor}_{i-1}^R(L, M) \rightarrow \cdots
	\]
	Again, by the Nakayama's Lemma, one obtains that
	\begin{equation}\label{Tor-equiv-mod-x-Tate}
		\widehat{\Tor}_i^R(L, M) = 0 \mbox{ for all } i \ll 0 \Longleftrightarrow \widehat{\Tor}_i^R(L, M/xM) = 0 \mbox{ for all } i \ll 0.
	\end{equation}
	Since $M/xM$ is CM of dimension $t-1$, by the induction hypothesis,
	\begin{equation}\label{Ext-Tor-vanishing-equiv-induction-step-Tate}
		\widehat{\Ext}_R^i(L, (M/xM)^{+}) = 0 \ \mbox{for all}\  i \ll 0 \Leftrightarrow \widehat{\Tor}_i^R(L, M/xM) = 0 \ \mbox{for all} \ i \ll 0.
	\end{equation}
	Putting together \eqref{Ext-equiv-mod-x-Tate}, \eqref{Tor-equiv-mod-x-Tate} and \eqref{Ext-Tor-vanishing-equiv-induction-step-Tate}, the desired equivalence follows.
\end{proof}

As a consequence of Lemma~\ref{lem:vanshing-ext-tor-Tate}, we obtain the desired result on symmetry in vanishing of Tate cohomologies.

\begin{theorem}\label{thm:Vanishing-Tate-coh}
	Along with Setup~\ref{setup:Gor-ring-M+}, let $M$ and $N$ be CM $R$-modules. Then the following are equivalent:
	\begin{enumerate}[\rm (1)]
		\item $\widehat{\Ext}_R^i(M,N^{+}) = 0$ for all $i \ll 0$.
		\item $\widehat{\Ext}_R^i(N,M^{+}) = 0$ for all $i \ll 0$.
		\item $\widehat{\Tor}_i^R(M, N) = 0$ for all $i \ll 0$.
	\end{enumerate}
\end{theorem}
%

\begin{proof}
	It is known that $\widehat{\Tor}_i^R(M, N) \cong \widehat{\Tor}_i^R(N, M)$ for all $i$, due to \cite[Thm.~3]{Ia07}, see also \cite[Thm.~3.7]{CJ14}. Hence the theorem follows from Lemma~\ref{lem:vanshing-ext-tor-Tate}.
\end{proof}

\section*{\bf Declarations}
{\bf Ethical approval:} The work in this article is original, and it is not submitted anywhere else for possible publication.

{\bf Competing interest:} None of the authors have any competing interests.

{\bf Authors' contributions:} Both the authors have equal contributions on each of the following: Investigation, Methodology, Writing the article and Review.

{\bf Funding:} Ghosh was supported by Startup Research Grant, SERB, DST, Govt.~of India, SRG/2020/000597.

{\bf Availability of data and materials:} This article has no associated data.

\end{document}